\documentclass[11pt]{article}
\usepackage[T1]{fontenc}
\usepackage{lmodern}
\usepackage{amsmath,amssymb,amsthm,mathtools}
\usepackage[margin=1in]{geometry}
\usepackage{microtype}
\usepackage{needspace}
\usepackage{cite}
\usepackage[colorlinks=true,linkcolor=blue,citecolor=blue,urlcolor=blue]{hyperref}
\numberwithin{equation}{section}
\newtheorem{theorem}{Theorem}[section]
\newtheorem{lemma}{Lemma}[section]
\newtheorem{proposition}[lemma]{Proposition}

\newtheorem*{conjecture}{Conjecture A}
\theoremstyle{definition}
\newtheorem{definition}[lemma]{Definition}
\theoremstyle{remark}
\newtheorem{remark}[lemma]{Remark}
\newcommand{\R}{\mathbb R}
\newcommand{\Sph}{\mathbb S}

\newcommand{\ran}{\operatorname{ran}}

\title{Finitely many solutions for the prescribed scalar curvature problem
\protect\\at the critical flatness order}
\date{}
\author{Wei Ji\thanks{School of Mathematics and Statistics, and  Key Laboratory of Nonlinear Analysis $\&$ Applications (Ministry of Education), Central China Normal University, P.O. Box 71010, Wuhan 430079, P. R. China. Email: \texttt{jiwei2020@amss.ac.cn}. }
\, and\, Yixing Pu\thanks{School of Mathematics and Statistics, and  Key Laboratory of Nonlinear Analysis $\&$ Applications (Ministry of Education), Central China Normal University, P.O. Box 71010, Wuhan 430079, P. R. China. Email: \texttt{yxpu@stu.ecnu.edu.cn}.
}
}

\begin{document}
\maketitle

\begin{abstract}
We study the following prescribed scalar curvature problem
\[
 -\Delta u=K(y)u^{2^*-1},\ \ u>0
 \,\ \ \text{in}\ \ \R^N,
\]
where $2^*=\frac{2N}{N-2}$.
For $N\geq4$, we construct a positive $K(y)$ with maximum
set $\{|y|=1\}$ and critical order of flatness $m=N-2$ at $|y|=1$,
such that the equation admits only finitely many positive solutions
in $D^{1,2}(\R^N)$. To this end, we prove the compactness of the solution set for this equation
by blow-up analysis and then apply Sard's theorem to choose $K$
for which each solution is non-degenerate and hence isolated,
so that the solution set is finite.
Our results also give a definite answer to Wei--Yan's conjecture in
[J. Wei and S. Yan, JFA, 2010].
\end{abstract}



\section{Introduction}

In this paper, we consider the following equation
\begin{equation}\label{eq:euclidean}
 \begin{cases}
 -\Delta u=K(y)u^{2^*-1}\,\ \ \text{in}\ \ \R^N,\\
 u>0, \ \ u\in D^{1,2}(\R^N),
 \end{cases}
\end{equation}
where $2^*=\frac{2N}{N-2}$, $N\geq4$, $K>0$, and $D^{1,2}(\R^N)$ is the completion of
$C_c^\infty(\R^N)$ with respect to the norm $\displaystyle \|\nabla u\|_{L^2(\R^N)}=\left(\int_{\R^N}|\nabla u|^2\,\mathrm{d}y\right)^{1/2}$.

Problem \eqref{eq:euclidean} is related to the prescribed scalar curvature
problem on the standard sphere $(\Sph^N,g_0)$, which is to find a metric
$g$ conformal to $g_0$ with scalar curvature $\mathbf K$.
In fact, taking $g=v^{4/(N-2)}g_0$, the prescribed scalar curvature problem is equivalent to finding a positive function $v$ on $\Sph^N$ satisfying the following equation
\begin{equation}\label{eq:spherical}
 -\Delta_{g_0}v+\frac{N(N-2)}4v
       =\frac{N-2}{4(N-1)}\mathbf K v^{2^*-1}
       \,\ \text{in }\Sph^N,\,\  v>0,\,\  v\in H^1(\Sph^N),
\end{equation}
where $\Delta_{g_0}$ denotes the Laplace--Beltrami operator with respect to
$g_0$. By stereographic projection $\pi_N$ and conformal invariance, we can deduce \eqref{eq:euclidean}
from \eqref{eq:spherical} with
$K(y)=\frac{N-2}{4(N-1)}\mathbf K(\pi_N^{-1}(y))$.

The prescribed scalar curvature problem has been extensively studied, 
see \cite{AAP,BC,Bianchi96,CNY,CGY93,CY91,CL98,CL01,DN85,LN88,Li95,Li96,Ni82,SZ96}
and the references therein. In particular,
Wei and Yan \cite[Theorem~1.1]{WY} proved that for $N\geq5$, if $K$ is
positive and radial, $K(y)=K(r)$, $r=|y|$, and has the expansion
\begin{equation}\label{eq:wy-expansion}
 \begin{gathered}
 K(r)=K(r_0)-c_0|r-r_0|^m+O(|r-r_0|^{m+\theta})
       \ \ \text{as}\ \ r\to r_0,\\
 r_0,c_0,\theta>0,\ \ 2\leq m<N-2,
 \end{gathered}
\end{equation}
then \eqref{eq:euclidean} admits infinitely many non-radial solutions.

They also proposed the following conjecture in general cases~\cite[Remark~1.5]{WY}.
\begin{conjecture}[Wei--Yan]
Assume that the set
\[
 \{y\in\R^N:\,K(y)=\max_{\R^N}K\}
\]
is an $l$-dimensional smooth manifold without boundary, where $1\leq l\leq
N-1$. Then problem \eqref{eq:euclidean} admits infinitely many solutions.
\end{conjecture}

The proof of \cite[Theorem~1.1]{WY} relies on the radial symmetry of $K$
in \eqref{eq:wy-expansion} and does not directly apply in the general
setting of Conjecture~A. The conjecture has remained open. Several
multiplicity results have been obtained under other assumptions on $K$ as follows.

Under suitable assumptions on a periodic function $K$, Li
\cite[Theorem~0.8]{Li95} obtained infinitely many solutions of
\eqref{eq:euclidean} for $N\geq3$, up to translations by integer
multiples of the period of $K$. The periodicity
assumption also plays an important role in the construction of Li, Wei and Xu
\cite{LWX}. They assume that $K\in C^1(\R^N)$ satisfies $0<c\leq K\leq C$ and
is periodic in the first $k$ variables, where $1\leq k<\frac{N-2}{2}$. Near the critical point $0$ of $K$, they assume
\[
 K(y)=K(0)+\sum_{j=1}^N a_j|y_j|^\beta+R(y),
       \ \ N-2<\beta<N,
\]
where $a_j\ne0$ and $\sum_j a_j<0$. Under suitable assumptions on $R$ and its
derivatives near $0$, they construct solutions with arbitrarily many bubbles.
Thus infinitely many solutions can also exist when $K$ is periodic and $\beta>N-2$.

Peng, Wang and Wei \cite{PWW} obtained infinitely many solutions for $N\geq5$
when $K$ is invariant under rotations of $y'\in\R^2$. They assume that $K$ is
bounded, nonnegative and of the form
\[
 K(y)=K(\rho,y''),\ \ 
 y=(y',y'')\in\R^2\times\R^{N-2},\ \ \rho=|y'|.
\]
They also assume that $K$ is $C^3$ near a point $(\rho_0,y_0'')$, with $\rho_0>0$
and $K(\rho_0,y_0'')=1$, and require
\[
 \begin{gathered}
 \nabla_{\rho,y''}K(\rho_0,y_0'')=0,\ \ 
 \deg\bigl(\nabla_{\rho,y''}K,(\rho_0,y_0'')\bigr)\ne0,\\
 \frac{\partial^2 K}{\partial\rho^2}(\rho_0,y_0'')
 +\sum_{j=3}^N\frac{\partial^2 K}{\partial y_j^2}(\rho_0,y_0'')<0.
 \end{gathered}
\]
Here $\nabla_{\rho,y''}K$ and its local Brouwer degree are taken in the $N-1$
variables $(\rho,y'')$. Under these assumptions, $K$ may depend on $y''$
without being even in those variables, and $(\rho_0,y_0'')$ may be a saddle
point of $K(\rho,y'')$.

Luo, Peng and Zhou \cite{LPZ} studied \eqref{eq:euclidean}
with $K=1+\varepsilon K_0$ for $N\geq5$, where
$K_0\in C^1(\R^N)\cap L^\infty(\R^N)$. They assume that $\Gamma$ is a
compact connected $C^2$ hypersurface without boundary consisting of
critical points of $K_0$, $K_0$ is $C^4$ near $\Gamma$, and
$\Delta K_0<0$ on $\Gamma$. Under suitable assumptions on $K_0$ and on
two distinct points $b_1,b_2\in\Gamma$, they construct, for every
$\varepsilon>0$ small enough, solutions whose two peaks concentrate at
$b_1$ and $b_2$ as $\varepsilon\to0$.

Recently, Li, Wei and Yang \cite{LWY} considered $K$ with the uniform expansion
\begin{equation}\label{eq:lwy-expansion}
 \begin{gathered}
 K(y)=K(r_0)-c_0\bigl||y|-r_0\bigr|^m
       +O\bigl(\bigl||y|-r_0\bigr|^{m+\theta}\bigr)
       \ \ \text{as}\ \ |y|\to r_0,\\
 r_0,c_0>0,\ \ \theta>2,\ \ 
 \min\left\{\frac{N-2}{2},2\right\}<m<N-2.
 \end{gathered}
\end{equation}
The value $K(r_0)$ is constant on $|y|=r_0$ and $K$ need not be
radial or invariant under rotations of $y'\in\R^2$. For $N\geq5$ such that
\[
 g''(x)<\frac{N-2}{N-1}\frac{(g'(x))^2}{g(x)}
       \ \ \text{for all}\ \ x\in(0,\pi),
\]
they obtain infinitely many solutions of \eqref{eq:euclidean}, where
$g(x)=\sum_{j=1}^{\infty}\frac{1-\cos(jx)}{j^N}$.

The flatness assumptions \eqref{eq:wy-expansion} and
\eqref{eq:lwy-expansion} require that the flatness order  $m<N-2$. It is natural to ask whether
the existence of infinitely many solutions of \eqref{eq:euclidean} still
follows at the critical flatness order $m=N-2$. \textbf{For $N\geq4$, we
construct a positive function $K$ with $m=N-2$, for which
\eqref{eq:euclidean} admits only finitely many solutions. Thus $K$ in
Theorem~\ref{thm:main} gives a counterexample to Conjecture~A.}

\Needspace{12\baselineskip}
Our main result is as follows.

\begin{theorem}\label{thm:main}
Suppose $N\geq4$ and $m=N-2$. Then there exist a positive function
$K\in C^2(\R^N)$ and a constant $c_0>0$ satisfying
\begin{equation}\label{eq:maxset}
 \max_{\R^N}K=2,\ \ \{y\in\R^N:\,K(y)=2\}=\{|y|=1\},
\end{equation}
and
\begin{equation}\label{eq:euclidean-expansion}
 K(y)=2-c_0\bigl||y|-1\bigr|^m
      +O\bigl(\bigl||y|-1\bigr|^{m+1}\bigr)
      \ \ \text{as}\ \ |y|\to1,
\end{equation}
such that the solution set of \eqref{eq:euclidean} is finite,
possibly empty.
\end{theorem}

\begin{remark}
The set $\{|y|=1\}$ in \eqref{eq:maxset} satisfies the hypothesis
of Conjecture~A with $l=N-1$. Thus Theorem~\ref{thm:main} gives a definite answer to Conjecture~A.
\end{remark}

\begin{remark}
Expansion \eqref{eq:euclidean-expansion} has the same form as the expansion in
\eqref{eq:lwy-expansion}, with $r_0=1$ and $\theta=1$.
The flatness order $m$ is critical, in the sense that $m=N-2$.
Theorem~\ref{thm:main} thus shows that the result of Li, Wei
and Yang \cite{LWY} on the existence of infinitely many solutions is, in some
sense, sharp with respect to the flatness order.
\end{remark}

\begin{remark}\label{rem:boundary}
The argument for Theorem~\ref{thm:main} can also be adapted to answer
a conjecture of Wang and Zhao \cite{WZ} on the prescribed boundary mean
curvature problem. More precisely, they considered
$u\in D^{1,2}(\R^N_+)$ satisfying
\begin{equation}\label{eq:boundary-problem}
 \left\{
 \begin{array}{@{}l@{\,\ \ }l@{}}
 -\Delta u=0,\ \ u>0
   &\text{in}\ \ \R^N_+,\\
 -\dfrac{2}{N-2}\partial_{y_N}u=H(y')u^{2^\#-1}
   &\text{on}\ \ \partial\R^N_+,
 \end{array}
 \right.
\end{equation}
where $2^\#=\frac{2(N-1)}{N-2}$ and
$\R^N_+=\{y=(y',y_N)\in\R^{N-1}\times\R:\,y_N>0\}$.
For $N\geq5$, they obtained infinitely many solutions when $H$ is
positive, bounded and radial and, with $r=|y'|$,
\[
 H(r)=H(r_0)-d_0|r-r_0|^m+O(|r-r_0|^{m+\theta}),
 \ \ r_0,d_0,\theta>0,\ \ 2\leq m<N-2.
\]
Meanwhile they conjectured that if the set
\[
 \left\{(y',0)\in\partial\R^N_+:\,
       H(y')=\max_{\R^{N-1}}H\right\}
\]
is an $l$-dimensional smooth manifold without boundary, where
$1\leq l\leq N-2$, then problem \eqref{eq:boundary-problem} admits
infinitely many positive solutions.
Combining the results in \cite{JLX} with arguments similar to those
in the proof of Theorem~\ref{thm:main}, we can then obtain an analogous
finiteness result for \eqref{eq:boundary-problem}.
\end{remark}

To construct $K$ in Theorem~\ref{thm:main}, we set $m=N-2$ and,
for $0<\varepsilon\leq1$, define the leading term $K_{\varepsilon}$ by
\begin{equation}\label{eq:background}
 K_{\varepsilon}(y)=2-\varepsilon\left[1-
       \left(\frac{4|y|^2}{(1+|y|^2)^2}\right)^{m/2}\right]^{m/2}.
\end{equation}
Theorem~\ref{thm:main} is an immediate consequence of the following theorem.

\begin{theorem}\label{thm:construction}
Suppose that $N\geq4$, $K_{\varepsilon}$ is given by \eqref{eq:background},
and $\Omega$ is a nonempty open subset of $\R^N$ satisfying
$\overline{\Omega}\subset\{y:\,1<|y|<2\}$.
Then there exists $\varepsilon_*\in(0,1]$ such that, for any
$0<\varepsilon\leq\varepsilon_*$, any integer $s\geq0$ and any $\eta>0$,
there exists $h\in C_c^\infty(\Omega)$ with $\|h\|_{C^s}<\eta$ such that
the function
\begin{equation*}
 K(y)=K_{\varepsilon}(y)+h(y)
\end{equation*}
is positive, belongs to $C^2(\R^N)$ and satisfies \eqref{eq:maxset} and
\eqref{eq:euclidean-expansion} in Theorem~\ref{thm:main}, with
$c_0=\varepsilon(m/2)^{m/2}$. The solution set of
\eqref{eq:euclidean} is finite, possibly empty.
\end{theorem}
\begin{remark}
The choice of $K_\varepsilon$ in \eqref{eq:background} is dictated by
the proof. In fact, the parameter $\varepsilon$ ensures that the
constants $L_i$ in Lemma~\ref{lem:li-global} tend to zero as
$\varepsilon_i\to0$.
Moreover, the choice of $2|y|/(1+|y|^2)$ gives
$K_\varepsilon(y)=K_\varepsilon(\frac{y}{|y|^2})$ for $y\ne0$,
so that the Kelvin transform \eqref{eq:kelvin-solution} reduces blow-up at infinity to blow-up at $0$.
\end{remark}

\begin{remark}\label{rem:scope}
For $N=3$, taking $m=N-2=1$ in \eqref{eq:background}, we obtain a
Lipschitz function $K_{\varepsilon}$ which does not satisfy the $C^1$
hypothesis in Li \cite[Theorem~4.2]{Li95}, and hence our argument for
Theorem~\ref{thm:construction} cannot be applied in this case.
However, our argument also applies to
the case $N=3$ with $m=2$. The conclusions of
Theorem~\ref{thm:construction} and hence of Theorem~\ref{thm:main}
therefore hold in the case $N=3,m=2$.
\end{remark}

\begin{remark}
If $N$ is even, then $m/2$ is an integer. Hence
$K_{\varepsilon}$ in \eqref{eq:background} and the function
$K=K_{\varepsilon}+h$ in Theorem~\ref{thm:construction}
belong to $C^\infty(\R^N)$.
Thus, for even $N\geq4$, Theorem~\ref{thm:main} holds with
$K\in C^\infty(\R^N)$.
\end{remark}

We divide the proof into three steps. We first consider possible blow-up
sequences $(u_i)$ of solutions to \eqref{eq:euclidean} with
$K=K_{\varepsilon_i}+h_i$ and $\varepsilon_i\to0^+$, and study
the blow-up points of $(u_i)$ and their properties. We next use the Kazdan--Warner identity to exclude the blow-up.
Applying Harnack's inequality and elliptic estimates, we then obtain
uniform a-priori estimates for $u$ and compactness of the solution set
of \eqref{eq:euclidean} for fixed $K=K_{\varepsilon}+h$. Finally, we fix $\varepsilon>0$ small enough and choose $h$ such that each
solution $u$ of \eqref{eq:euclidean} with $K=K_{\varepsilon}+h$ is non-degenerate
in $D^{1,2}(\R^N)$. To this end, we study the linearized operator of \eqref{eq:euclidean}
around a solution $u$ and apply Sard's theorem to the projection
$(h,u)\mapsto h$ on a finite-dimensional solution manifold. 
Once non-degeneracy is established, we apply the implicit
function theorem to obtain local uniqueness at each solution $u$.
Together with compactness, we obtain that the solution set of \eqref{eq:euclidean} is finite.

This paper is organized as follows. Section~2 is devoted to the
properties of $K$, the regularity of solutions of
\eqref{eq:euclidean}, and the Kazdan--Warner identity.
In Section~3, we establish a-priori estimates by blow-up analysis.
In Section~4, we study the linearized equation and choose a suitable $h$ to complete the proof of
Theorem~\ref{thm:construction}.

\section{Preliminaries}

In the rest of the paper, $N$ and $\Omega$ are fixed as in Theorem~\ref{thm:construction}.
The positive constants $c$ and $C$ may change from line to line. We denote by
$\mathbf{B}_\rho(a)$ the ball $\{y\in\R^N:\,|y-a|<\rho\}$ and put
$\mathbf{B}_\rho=\mathbf{B}_\rho(0)$. Here $D^j f$ denotes all partial
derivatives of $f$ of order $j$.

To describe $K$ for large $|y|$, define
\[
 \tilde{K}(y)=K\left(\frac{y}{|y|^2}\right),\ \ 
 \tilde{h}(y)=h\left(\frac{y}{|y|^2}\right),\ \ y\ne0.
\]
For $h\in C_c^\infty(\Omega)$, we set $\tilde{h}(0)=0$. Thus $\tilde{K}$ near
$0$ corresponds to $K$ for large $|y|$. We use the same notation for the
continuous extension of $\tilde{K}$ across $0$. It follows from
\eqref{eq:background} that
\begin{equation}\label{eq:kelvin}
 K_{\varepsilon}\left(\frac{y}{|y|^2}\right)=K_{\varepsilon}(y),
       \ \ y\ne0.
\end{equation}
We first give some properties of $K_{\varepsilon}$.

\begin{lemma}\label{lem:background-properties}
Suppose $0<\varepsilon\leq1$. Then $K_{\varepsilon}\in C^2(\R^N)\cap
C^{m-1,1}(\R^N)$ and satisfies
\[
 \begin{aligned}
 &\max_{\R^N}K_{\varepsilon}=2,\ \ 
       \{y\in\R^N:\,K_{\varepsilon}(y)=2\}=\{|y|=1\},\\
 &\min_{\R^N}K_{\varepsilon}=2-\varepsilon,\ \ 
       \{y\in\R^N:\,K_{\varepsilon}(y)=2-\varepsilon\}=\{0\},\\
&\lim_{|y|\to\infty}K_{\varepsilon}(y)=2-\varepsilon.
 \end{aligned}
\]
Moreover,
\[
 \{y\in\R^N:\,\nabla K_{\varepsilon}(y)=0\}=\{0\}\cup\{|y|=1\},
\]
and there exist a fixed open neighborhood $\mathcal U$ of $\{0\}\cup\{|y|=1\}$ and a constant
$C>0$ such that, for $y\in\mathcal U$ with $\nabla
K_{\varepsilon}(y)\ne0$,
\begin{equation}\label{eq:flatness}
 |D^jK_{\varepsilon}(y)|
 \leq C\varepsilon^{(j-1)/(m-1)}
       |\nabla K_{\varepsilon}(y)|^{(m-j)/(m-1)}\,\ \  \text{for any}\ \ 2\leq j\leq m.
\end{equation}
\end{lemma}

\begin{proof}
By \eqref{eq:background}, $K_{\varepsilon}\in C^2(\R^N)\cap C^{m-1,1}(\R^N)$,
$K_{\varepsilon}(0)=2-\varepsilon$, $K_{\varepsilon}(y)=2$ for $|y|=1$, and
$\lim_{|y|\to\infty}K_{\varepsilon}(y)=2-\varepsilon$.
Set $r=|y|$. We calculate that $\nabla K_{\varepsilon}(0)=0$,
$K_{\varepsilon}'(r)>0$ for $0<r<1$ and $K_{\varepsilon}'(r)<0$ for $r>1$.

As $r\to0$, we have
\begin{equation}\label{eq:polar-expansion}
 K_{\varepsilon}
 =2-\varepsilon+\varepsilon m2^{m-1}r^m
       +O(\varepsilon r^{m+2}),\ \,
 K_{\varepsilon}'(r)
 =\varepsilon m^2 2^{m-1}r^{m-1}
       +O(\varepsilon r^{m+1}),
\end{equation}
and
\begin{equation}\label{eq:polar-derivative-bounds}
 \begin{aligned}
 c\varepsilon r^{m-1}&\leq|\nabla K_{\varepsilon}|
       \leq C\varepsilon r^{m-1},\\
 |D^jK_{\varepsilon}|&\leq C\varepsilon r^{m-j}
      \,\ \  \text{for any}\ \ 2\leq j\leq m.
 \end{aligned}
\end{equation}
As $r\to1$, we have
\begin{equation}\label{eq:equator-expansion}
 \begin{aligned}
 K_{\varepsilon}
 &=2-\varepsilon\left(\frac m2\right)^{m/2}|r-1|^m
       +O(\varepsilon|r-1|^{m+1}),\\
 K_{\varepsilon}'(r)
 &=-\varepsilon m\left(\frac m2\right)^{m/2}|r-1|^{m-2}(r-1)
       +O(\varepsilon|r-1|^m),
 \end{aligned}
\end{equation}
and
\begin{equation}\label{eq:equator-derivative-bounds}
 \begin{gathered}
 c\varepsilon|r-1|^{m-1}\leq|\nabla K_{\varepsilon}|
       \leq C\varepsilon|r-1|^{m-1},\\
 |D^jK_{\varepsilon}|\leq C\varepsilon|r-1|^{m-j}
      \,\ \  \text{for any}\ \ 2\leq j\leq m.
 \end{gathered}
\end{equation}
We thus deduce from \eqref{eq:polar-derivative-bounds}
and \eqref{eq:equator-derivative-bounds} that \eqref{eq:flatness} holds. This
completes the proof.
\end{proof}

\begin{lemma}\label{lem:supported-properties}
Suppose that $\varepsilon>0$ is small enough and that
\begin{equation}\label{eq:admissible}
 K=K_{\varepsilon}+h,\ \
 h\in C_c^\infty(\Omega),\ \ \|h\|_{C^m}\leq\varepsilon^2.
\end{equation}
Then both $K$ and $\tilde{K}$ belong to
$C^2(\R^N)\cap C^{m-1,1}(\R^N)$.
The critical set of $K$ in $\R^N$ is $\{0\}\cup\{|y|=1\}$,
$1/2\leq K\leq5/2$, and \eqref{eq:maxset} holds.

Moreover, there exist constants $\gamma,C>0$ and an open neighborhood
$\mathcal V\subset\mathbf B_2$ of $\{0\}\cup\{|y|=1\}$,
all independent of $\varepsilon$ and $h$, such that
\[
 K(0)=\tilde{K}(0)=2-\varepsilon,\ \ 
 K=\tilde{K}=K_{\varepsilon}\,\ \ \text{on}\ \ \mathcal V,
\]
\begin{equation}\label{eq:small-gradient-region}
 \begin{split}
 &\{y\in\overline{\mathbf B}_2:\,|\nabla K(y)|<\gamma\varepsilon\}\cup\{y\in\overline{\mathbf B}_2:\,
       |\nabla\tilde{K}(y)|<\gamma\varepsilon\}
       \subset\mathcal V,
 \end{split}
\end{equation}
and
\[
 \|K-2\|_{C^{m-1,1}(\overline{\mathbf B}_3)}
 +\|\tilde{K}-2\|_{C^{m-1,1}(\overline{\mathbf B}_3)}\leq C\varepsilon.
\]
For $2\leq j\leq m$, the following estimates hold on $\mathbf B_3$ wherever the
corresponding gradient is nonzero:
\begin{equation}\label{eq:perturbed-flatness}
 \begin{aligned}
 |D^jK|&\leq C\varepsilon^{(j-1)/(m-1)}
       |\nabla K|^{(m-j)/(m-1)},\\
 |D^j\tilde{K}|&\leq C\varepsilon^{(j-1)/(m-1)}
       |\nabla\tilde{K}|^{(m-j)/(m-1)}.
 \end{aligned}
\end{equation}
Equivalently, if $\mathbf{F}=(K-2)/\varepsilon$ and
$\tilde{\mathbf{F}}=(\tilde{K}-2)/\varepsilon$, then for any $2\leq j\leq m$,
\begin{equation}\label{eq:normalized-flatness}
      \begin{aligned}
             |D^j\mathbf{F}|\leq C|\nabla\mathbf{F}|^{(m-j)/(m-1)} \ \,\ \ &\text{on}\ \ 
\{y\in\mathbf B_3:\,\nabla\mathbf{F}(y)\ne0\}\\
 |D^j\tilde{\mathbf{F}}|
 \leq C|\nabla\tilde{\mathbf{F}}|^{(m-j)/(m-1)} \,\ \ &\text{on}\ \ \{y\in\mathbf B_3:\,\nabla\tilde{\mathbf{F}}(y)\ne0\}
      \end{aligned}
\end{equation}
\end{lemma}

\begin{proof}
By the definition of $\tilde{h}$,
\[
 \operatorname{supp}\tilde{h}
 \subset\left\{\frac{y}{|y|^2}:\,y\in\overline\Omega\right\}
 \subset\{1/2<|y|<1\}.
\]
Thus
\[
 \tilde{h}\in C_c^\infty(\R^N),\ \ 
 \|\tilde{h}\|_{C^m}\leq C\|h\|_{C^m}\leq C\varepsilon^2.
\]
With $\mathcal U$ as in Lemma~\ref{lem:background-properties}, choose
$\mathcal V$ such that
\[
 \begin{gathered}
 \{0\}\cup\{|y|=1\}\subset\mathcal V,\ \ 
 \overline{\mathcal V}\subset\mathcal U\cap\mathbf B_2, \ \ 
 \overline{\mathcal V}\cap
 \left(\overline\Omega\cup
 \left\{\frac{y}{|y|^2}:\,y\in\overline\Omega\right\}\right)=\varnothing.
 \end{gathered}
\]
It follows that $h=\tilde{h}=0$ and $K=\tilde{K}=K_{\varepsilon}$
on $\mathcal V$.

The function $(K_{\varepsilon}-2)/\varepsilon$ is independent of $\varepsilon$.
By Lemma~\ref{lem:background-properties}, $\min_{\overline{\mathbf B}_3\setminus\mathcal V}
       \frac{|\nabla K_{\varepsilon}|}{\varepsilon}\geq c>0$.
Together with $\|h\|_{C^m}+\|\tilde{h}\|_{C^m}\leq C\varepsilon^2$,
we have, for $\varepsilon>0$ small enough,
\[
 |\nabla K(y)|\geq\frac{c\varepsilon}{2},\ \ 
 |\nabla\tilde{K}(y)|\geq\frac{c\varepsilon}{2},
 \ \ y\in\overline{\mathbf B}_3\setminus\mathcal V.
\]
Taking $0<\gamma<c/2$, we then obtain that \eqref{eq:small-gradient-region} holds.
Since $K=K_{\varepsilon}$ on
$\mathcal V\cup(\R^N\setminus\mathbf B_3)$, we also obtain that
\[
 \{y\in\R^N:\,\nabla K(y)=0\}=\{0\}\cup\{|y|=1\}.
\]

It follows from \eqref{eq:background}, \eqref{eq:kelvin} and \eqref{eq:admissible} that 
\[
 \begin{aligned}
 &\|K-2\|_{C^{m-1,1}(\overline{\mathbf B}_3)}
 +\|\tilde{K}-2\|_{C^{m-1,1}(\overline{\mathbf B}_3)}\\
 &\quad\leq
 2\|K_{\varepsilon}-2\|_{C^{m-1,1}(\overline{\mathbf B}_3)}
 +C\|h\|_{C^m}+C\|\tilde{h}\|_{C^m}
 \leq C\varepsilon.
 \end{aligned}
\]
For $2\leq j\leq m$ and $y\in\overline{\mathbf B}_3\setminus\mathcal V$,
we deduce from \eqref{eq:background} and \eqref{eq:admissible} that $|D^jK(y)|+|D^j\tilde{K}(y)|\leq C\varepsilon$.
Since $|\nabla K(y)|,|\nabla\tilde{K}(y)|\geq c\varepsilon/2$, we obtain that
\[
 \begin{aligned}
 |D^jK(y)|
 &\leq C\varepsilon
 \leq C\varepsilon^{(j-1)/(m-1)}
       |\nabla K(y)|^{(m-j)/(m-1)},\\
 |D^j\tilde{K}(y)|
 &\leq C\varepsilon
 \leq C\varepsilon^{(j-1)/(m-1)}
       |\nabla\tilde{K}(y)|^{(m-j)/(m-1)}.
 \end{aligned}
\]
On $\mathcal V$, $K=\tilde{K}=K_{\varepsilon}$, so that
we deduce from \eqref{eq:flatness} that \eqref{eq:perturbed-flatness} holds.
Substituting $K=2+\varepsilon\mathbf{F}$ and $\tilde{K}=2+\varepsilon\tilde{\mathbf{F}}$
in \eqref{eq:perturbed-flatness}, we then derive that \eqref{eq:normalized-flatness} holds.

Finally, the positive number $\min_{\overline\Omega}\frac{2-K_{\varepsilon}}{\varepsilon}$
is independent of $\varepsilon$. For $\varepsilon>0$ small enough,
\[
 \begin{aligned}
 K&\leq\max_{\overline\Omega}K_{\varepsilon}+\varepsilon^2<2
       \,\ \ \text{on}\ \ \overline\Omega,\\
 K&=K_{\varepsilon}
       \,\ \ \text{on}\ \ \R^N\setminus\Omega.
 \end{aligned}
\]
Hence \eqref{eq:maxset} holds. Moreover,
\[
 \frac12\leq2-\varepsilon-\varepsilon^2
       \leq K\leq2+\varepsilon^2\leq\frac52
       \,\ \ \text{in}\ \ \R^N.
\]
This completes the proof.
\end{proof}

Next, we give a regularity result for solutions of \eqref{eq:euclidean}.

\begin{lemma}\label{lem:regularity}
Let $K$ satisfy \eqref{eq:admissible} and let $u$ be a solution of
\eqref{eq:euclidean}. Denote its Kelvin transform by
\begin{equation}\label{eq:kelvin-solution}
 \widehat u(y)=|y|^{2-N}u\left(\frac{y}{|y|^2}\right),
 \ \ y\ne0.
\end{equation}
Then $\widehat u\in D^{1,2}(\R^N)$ and $-\Delta\widehat u=\tilde{K}\widehat u^{2^*-1}$ in $\R^N$. Moreover, for any $0<\alpha<1$,
\begin{equation}\label{eq:solution-regularity}
 u,\widehat u\in C^{m+1,\alpha}_{\rm loc}(\R^N),\ \ 
 u,\widehat u>0\,\ \ \text{in}\ \ \R^N.
\end{equation}
\end{lemma}

\begin{proof}
It follows from Lemma~\ref{lem:supported-properties} that $K,\tilde{K}\in
C^{m-1,1}(\R^N)$ and $K,\tilde{K}>0$ in $\R^N$. Since $u,\widehat u\in
L^{2^*}(\R^N)$, we apply a bootstrap argument as in Brezis--Kato \cite[proof of
Theorem~2.3(b)]{BK} to obtain that $u,\widehat u\in L^q_{\rm loc}(\R^N)$ for
all $1\leq q<\infty$ and $u,\widehat u\in
C^{2,\alpha}_{\rm loc}(\R^N)$. Since $K,\tilde{K}\in C^{m-1,\alpha}_{\rm
loc}(\R^N)$, we apply the bootstrap argument and the strong maximum principle to
obtain that \eqref{eq:solution-regularity} holds. This completes the proof.
\end{proof}

Finally, we recall the Kazdan--Warner identity for \eqref{eq:euclidean}. Define
for $|\bar y|=1$,
\begin{equation}\label{eq:equatorial-field}
 Y_{\bar y}(y)=\frac{1+|y|^2}{2}\bar y-(\bar y\cdot y)y.
\end{equation}

\begin{lemma}\label{lem:kw}
Let $K$ satisfy \eqref{eq:admissible} and let $u$ be a solution of
\eqref{eq:euclidean}. Then for $Y(y)\equiv
\tilde C\in\R^N$, $Y(y)=y$ or $Y=Y_{\bar y}$, the following Kazdan--Warner
identities hold:
\begin{equation}\label{eq:KW}
 \int_{\R^N}(Y\cdot\nabla K)u^{2^*}\,\mathrm{d}y=0,\ \ 
 \int_{\R^N}(Y\cdot\nabla\tilde{K})\widehat u^{2^*}\,\mathrm{d}y=0.
\end{equation}
\end{lemma}

\begin{proof}
It follows from Lemma~\ref{lem:regularity} and \eqref{eq:kelvin-solution} that
\[
 u(y)=O(|y|^{2-N}),\ \ \nabla u(y)=O(|y|^{1-N})
       \ \ \text{as}\ \ |y|\to\infty.
\]
Moreover, $Y$ satisfies
\[
 \partial_jY_k+\partial_kY_j
       =\frac2N(\operatorname{div}Y)\delta_{jk},\ \ 
 \Delta(\operatorname{div}Y)=0
\]
for $Y(y)\equiv\tilde C\in\R^N$, $Y(y)=y$ or $Y=Y_{\bar y}$.
Multiplying \eqref{eq:euclidean} by $Y\cdot\nabla
u+\frac{N-2}{2N}(\operatorname{div}Y)u$ on $\mathbf B_R$ and integrating by
parts, we then obtain that
\[
 \int_{\mathbf B_R}(Y\cdot\nabla K)u^{2^*}\,\mathrm{d}y=O(R^{3-N}),
\]
since $Y(y)=O(|y|^2)$ and $\nabla(\operatorname{div}Y)(y)=O(1)$ as
$|y|\to\infty$.
Letting $R\to\infty$, we obtain the first identity in \eqref{eq:KW}. Applying the
same calculation to $\widehat u$ with $\tilde{K}$, we then obtain the second
identity in \eqref{eq:KW}.
\end{proof}

\section{A-priori estimates}

For $K$ satisfying \eqref{eq:admissible}, we prove uniform a-priori estimates
for all solutions $u$ of \eqref{eq:euclidean} when $\varepsilon$ is small enough.
To this end, we study blow-up sequences, beginning with the definitions and
results from Li \cite{Li95,Li96}.

For an open set $\mathcal O\subset\R^N$, consider
\begin{equation}\label{eq:local-critical}
 -\Delta u_i=K_i u_i^{2^*-1},\ \ u_i>0\,\ \ \text{in}\ \ \mathcal O,
\end{equation}
where $u_i\in C^2(\mathcal O)$ and $K_i\in C^1(\mathcal O)$ with $K_i>0$ in
$\mathcal O$.

\begin{definition}\label{def:blowup}
\leavevmode
We recall the definitions in (i)--(iii) from Li
\cite[Definitions~0.1--0.3]{Li95}.
\begin{enumerate}
\renewcommand{\labelenumi}{(\roman{enumi})}

\item
A point $\bar y\in\mathcal O$ is a blow-up point of $(u_i)$
if there are points $y_i\to\bar y$ such that $u_i(y_i)\to\infty$
as $i\to\infty$.

\item
A point $\bar y\in\mathcal O$ is an isolated blow-up point
of $(u_i)$ if there are local maximum points $y_i\to\bar y$ with
$u_i(y_i)\to\infty$ as $i\to\infty$ and constants $\bar r,\bar C>0$
independent of $i$ such that
$\bar r<\operatorname{dist}(\bar y,\partial\mathcal O)$ and
\begin{equation}\label{eq:isolated-blowup}
 u_i(y)\leq\bar C|y-y_i|^{-(N-2)/2}
 \ \ \text{for}\ \ 0<|y-y_i|<\bar r.
\end{equation}

\item
An isolated blow-up point $\bar y$ is simple if there is
$\rho>0$ independent of $i$ such that
$r^{(N-2)/2}\overline u_i(r)$ has precisely one critical point in
$(0,\rho)$ for all large $i$, where
\[
 \overline u_i(r):=\frac{1}{|\partial\mathbf B_r|}
                  \int_{\partial\mathbf B_r(y_i)}u_i\,\mathrm{d}S,
\]
and $y_i$ are the local maximum points in (ii).

\item
We extend the definitions in (i)--(iii) to $u_i$ at $\infty$ by applying them
to $(\widehat u_i)$ at $0$.
\end{enumerate}
\end{definition}

\Needspace{6\baselineskip}
We also use the following flatness condition $(*)_\beta$ from Li \cite[Definition~0.4]{Li95}.

\begin{definition}\label{def:li-flatness}
Let $\beta\geq2$ be an integer and $f\in C^{\beta-1,1}(\mathcal O)$. We say that $f$ satisfies condition $(*)_\beta$
in $\mathcal O$ if there is a constant $L>0$ such that
$\|\nabla f\|_{C^0(\mathcal O)}\leq L$ and
\begin{align*}
 |D^jf(y)|\leq L|\nabla f(y)|^{(\beta-j)/(\beta-1)},\ \ 
 2\leq j\leq\beta,\ \ y\in\mathcal O,\ \ \nabla f(y)\ne0.
\end{align*}
\end{definition}

We recall the following estimates from Li
\cite[Propositions~1.5--1.6]{Li96}.

\begin{lemma}\label{lem:li-estimates}
Suppose that $u_i$ solves \eqref{eq:local-critical} in $\mathbf B_{2r}(\bar y)$,
where $r>0$ and
\begin{equation}\label{eq:li-local-bounds}
 K_i\geq A^{-1},\ \ 
 \|K_i\|_{C^2(\mathbf B_{2r}(\bar y))}\leq A
\end{equation}
with $A>0$ independent of $i$. Let $y_i\to\bar y$ be local maximum points with
$M_i=u_i(y_i)\to\infty$ as $i\to\infty$, and suppose that
\[
 u_i(y)\leq\bar C|y-y_i|^{-(N-2)/2}
 \ \ \text{for}\ \ y\in\mathbf B_{2r}(\bar y)\setminus\{y_i\},
\]
where $\bar C>0$ is independent of $i$. Suppose also that $\bar y$ is a simple
blow-up point. Then, for all large $i$ and $0<|y-y_i|\leq r$,
\begin{equation}\label{eq:bubble-bounds}
 cM_i\bigl(1+CM_i^{4/(N-2)}|y-y_i|^2\bigr)^{-(N-2)/2}
 \leq u_i(y)\leq CM_i^{-1}|y-y_i|^{2-N},
\end{equation}
where $c,C>0$ are independent of $i$.
\end{lemma}

By stereographic projection and conformal invariance, we obtain the
following result from Li \cite[Remark~0.1, Proposition~3.1 and Theorem~4.2]{Li95}.

\begin{lemma}\label{lem:li-global}
Let $N\geq4$. Suppose that $u_i\in D^{1,2}(\R^N)$ solves
\eqref{eq:local-critical} in $\R^N$ and that
$K_i,\tilde{K}_i\in C^2(\R^N)$ satisfy
\begin{equation}\label{eq:li-global-bounds}
 \begin{gathered}
 K_i,\tilde{K}_i\geq A^{-1}\,\ \ \text{in}\ \ \mathbf B_3,\\
 \|K_i\|_{C^2(\mathbf B_3)}
 +\|\tilde{K}_i\|_{C^2(\mathbf B_3)}\leq A.
 \end{gathered}
\end{equation}
where $A>0$ is independent of $i$. If both $(K_i)$ and
$(\tilde{K}_i)$ satisfy $(*)_{N-2}$ in $\mathbf B_3$ with constants
$L_i>0$ such that $L_i\to0$ as $i\to\infty$, then, up to a
subsequence, either
\[
 \sup_{y\in\R^N}(1+|y|^2)^{(N-2)/2}u_i(y)\leq C
\]
with $C>0$ independent of $i$, or $(u_i)$ has precisely one blow-up point
$\bar y\in\R^N\cup\{\infty\}$, which is isolated and simple.

\end{lemma}

We now apply Lemmas~\ref{lem:li-estimates} and~\ref{lem:li-global} to describe
a blow-up sequence. Let $\varepsilon_i\to0$ as $i\to\infty$ and let
\[
 K_i=K_{\varepsilon_i}+h_i,\ \ 
 h_i\in C_c^\infty(\Omega),\ \ 
 \|h_i\|_{C^m}\leq\varepsilon_i^2,
\]
as in \eqref{eq:admissible}. Let $u_i$ solve \eqref{eq:euclidean} with
coefficient $K_i$.

\begin{proposition}\label{prop:blowup-sequence}
Suppose that
\[
 \sup_{y\in\R^N}(1+|y|^2)^{(N-2)/2}u_i(y)\to\infty
 \ \ \text{as}\ \ i\to\infty.
\]
Up to a subsequence, $(u_i)$ has precisely one blow-up point
$\bar y\in\R^N\cup\{\infty\}$, which is isolated and simple. After replacing
$(u_i,K_i)$ by $(\widehat u_i,\tilde{K}_i)$ when $\bar y=\infty$, we may
assume that $\bar y\in\R^N$. There are local maximum points $y_i\to\bar y$ such that
\begin{equation*}
 M_i=u_i(y_i)\to\infty
 \ \ \text{as}\ \ i\to\infty
\end{equation*}
and \eqref{eq:bubble-bounds} holds for some fixed $r>0$. For any fixed
neighborhood $\mathtt{V}$ of $\bar y$,
\begin{equation}\label{eq:global-tail-bound}
 u_i(y)\leq C_{\mathtt{V}} M_i^{-1}(1+|y|^2)^{-(N-2)/2}
 \, \ \text{for any}\ \ y\in\R^N\setminus \mathtt{V},
\end{equation}
and consequently
\begin{equation}\label{eq:tail-integral}
 \int_{\R^N\setminus \mathtt{V}}u_i^{2^*}\,\mathrm{d}y\leq C_{\mathtt{V}} M_i^{-2^*}.
\end{equation}
The constants in these estimates are independent of $i$.
\end{proposition}

\begin{proof}
By Lemma~\ref{lem:supported-properties}, $K_i$ and $\tilde{K}_i$ satisfy
\eqref{eq:li-global-bounds}. Moreover, using the bounds for $\nabla K_i,\nabla\tilde{K}_i$, we then deduce from
\eqref{eq:perturbed-flatness} that $(*)_m$ holds in $\mathbf B_3$ with
$L_i=C\varepsilon_i^{1/(m-1)}$.
Since $m=N-2$ and $\varepsilon_i\to0$, Lemma~\ref{lem:li-global} yields a
unique isolated simple blow-up point $\bar y$. Apply the Kelvin transform
if $\bar y=\infty$, and choose the local maximum points $y_i\to\bar y$
in Definition~\ref{def:blowup}.
By \eqref{eq:li-global-bounds} and the Kelvin transform, $K_i$ satisfies
\eqref{eq:li-local-bounds} on any fixed ball $\mathbf B_{2r}(\bar y)$,
with $A$ depending on $r$ and $\bar y$. For $r>0$ small, \eqref{eq:isolated-blowup}
holds on $\mathbf B_{2r}(\bar y)$, so that we can apply
Lemma~\ref{lem:li-estimates} to obtain that \eqref{eq:bubble-bounds} holds.

It remains to prove \eqref{eq:global-tail-bound}.
Using the upper bound in \eqref{eq:bubble-bounds}, we then obtain that $u_i\leq CM_i^{-1}$
on $\{r/4\leq|y-\bar y|\leq r/2\}$ for all large $i$.
Since the blow-up point in
$\R^N\cup\{\infty\}$ is unique, we have uniform bounds for $u_i$ on compact subsets of
$\R^N\setminus\{\bar y\}$ and for $\widehat u_i$ near $0$.
On any fixed ball $\overline{\mathbf B}_{2\rho}(\xi)\subset
\R^N\setminus\{\bar y\}$, the functions $K_i u_i^{2^*-2}$ are uniformly
bounded. Applying Harnack's inequality to
\eqref{eq:local-critical}, we obtain that
\begin{equation}\label{eq:tail-harnack}
 \sup_{\mathbf B_\rho(\xi)}u_i
       \leq C\inf_{\mathbf B_\rho(\xi)}u_i,
\end{equation}
where $C$ is independent of $i$. If another fixed ball satisfies
$\overline{\mathbf B}_{2\sigma}(\eta)\subset\R^N\setminus\{\bar y\}$ and
$\mathbf B_\rho(\xi)\cap\mathbf B_\sigma(\eta)\ne\varnothing$, then
\begin{equation}\label{eq:overlap-harnack}
 \sup_{\mathbf B_\sigma(\eta)}u_i
       \leq C\inf_{\mathbf B_\sigma(\eta)}u_i
       \leq C\sup_{\mathbf B_\rho(\xi)}u_i.
\end{equation}
Let $\mathcal E\subset\R^N\setminus\{\bar y\}$ be a fixed compact set.
Choose finitely many fixed balls $\mathbf B_{\rho_j}(\xi_j)$,
$1\leq j\leq J$, covering $\mathcal E$, each intersecting the next, such that
\[
 \mathbf B_{\rho_1}(\xi_1)\subset\{r/4<|y-\bar y|<r/2\},\ \ 
 \overline{\mathbf B}_{2\rho_j}(\xi_j)\subset\R^N\setminus\{\bar y\}
       \ \ (1\leq j\leq J).
\]
Since $u_i\leq CM_i^{-1}$ on $\mathbf B_{\rho_1}(\xi_1)$, we apply
\eqref{eq:tail-harnack} and \eqref{eq:overlap-harnack} successively to
$\mathbf B_{\rho_1}(\xi_1),\ldots,\mathbf B_{\rho_J}(\xi_J)$ to obtain that
$\sup_{\mathcal E}u_i\leq C_{\mathcal E}M_i^{-1}$, where
$C_{\mathcal E}$ is independent of $i$.

Choose $R>2(1+|\bar y|)$. The functions $\widehat u_i$ and
$\tilde{K}_i\widehat u_i^{2^*-2}$ are uniformly bounded on $\mathbf B_{2/R}$.
Harnack's inequality yields that
\[
 \sup_{\mathbf B_{1/R}}\widehat u_i
       \leq C\inf_{\mathbf B_{1/R}}\widehat u_i
       \leq CM_i^{-1},
\]
where the last inequality follows from $\sup_{\mathcal E}u_i\leq C_{\mathcal E}M_i^{-1}$ with $\mathcal E=\partial \mathbf B_{2R}$
and \eqref{eq:kelvin-solution}. Thus
$u_i(y)\leq CM_i^{-1}|y|^{2-N}$ for $|y|\geq R$.
Together with the bound for $u_i$ on
$\overline{\mathbf B}_R\setminus \mathtt{V}$, this proves
\eqref{eq:global-tail-bound}.
Since $\int_{\R^N}(1+|y|^2)^{-N}\,\mathrm{d}y<\infty$,
\eqref{eq:tail-integral} follows.
\end{proof}

We next use the Kazdan--Warner identity \eqref{eq:KW} to show that the blow-up
point $\bar y$ in Proposition~\ref{prop:blowup-sequence} satisfies
$\bar y=0$ or $|\bar y|=1$, and to derive the local integral estimate
\eqref{eq:kw-local-upper}. Set
\[
 \begin{aligned}
 \mathbf{F}_i(y)&=\frac{K_i(y)-2}{\varepsilon_i}
       =\frac{K_{\varepsilon_i}(y)-2+h_i(y)}{\varepsilon_i},\\
 \tilde{\mathbf{F}}_i(y)&=\frac{\tilde{K}_i(y)-2}{\varepsilon_i}
       =\frac{K_{\varepsilon_i}(y)-2+\tilde{h}_i(y)}{\varepsilon_i}.
 \end{aligned}
\]
After a Kelvin transform, we replace $(u_i,K_i,\mathbf{F}_i)$ by
$(\widehat u_i,\tilde K_i,\tilde{\mathbf{F}}_i)$ and keep the notation
$(u_i,K_i,\mathbf{F}_i)$.

\begin{lemma}\label{lem:kw-local}
Let $\bar y\in\R^N$, $y_i$ and $M_i$ be as in Proposition~\ref{prop:blowup-sequence}.
Then $\bar y=0$ or $|\bar y|=1$. Write $t=(t_1,\ldots,t_N)$.
If $\bar y=0$, take $Y(y)=y$ and $t(y)=y$. If
$|\bar y|=1$, take $Y=Y_{\bar y}$ defined by \eqref{eq:equatorial-field} and
smooth local coordinates $t=t(y)$ with $t(\bar y)=0$ and $t_N(y)=|y|-1$. Then
there is a fixed $\rho>0$ such that $\mathtt{V}_\rho=\{y:\,|t(y)|<\rho\}\subset\mathcal V$,
the estimates \eqref{eq:bubble-bounds} hold on $\mathtt{V}_\rho$ for all large $i$, and
\[
 K_i=K_{\varepsilon_i},\ \ 
 Y\cdot\nabla K_i\geq c\varepsilon_i|t_N(y)|^m\geq0
 \,\ \ \text{on}\ \ \mathtt{V}_\rho.
\]
Moreover,
\begin{equation}\label{eq:kw-local-upper}
 0\leq\int_{\mathtt{V}_\rho}(Y\cdot\nabla K_i)u_i^{2^*}\,\mathrm{d}y
       \leq C\varepsilon_i M_i^{-2^*},
\end{equation}
where $c,C>0$ are independent of $i$.
\end{lemma}

\begin{proof}
We first prove that $\bar y=0$ or $|\bar y|=1$.
By \eqref{eq:admissible} and Lemma~\ref{lem:supported-properties},
\[
 \|h_i\|_{C^m}+\|\tilde{h}_i\|_{C^m}\leq C\varepsilon_i^2.
\]
It follows from \eqref{eq:background} that $\mathbf{F}_i\to \mathbf{F}_*$ in
$C^1_{\rm loc}(\R^N)$, where
\[
 \mathbf{F}_*(y)=-\left[1-
       \left(\frac{4|y|^2}{(1+|y|^2)^2}\right)^{m/2}\right]^{m/2}.
\]
By Lemma~\ref{lem:background-properties}, we have
\[
 \{y\in\R^N:\,\nabla \mathbf{F}_*(y)=0\}
       =\{0\}\cup\{|y|=1\}.
\]

For each $Y$ in Lemma~\ref{lem:kw},
\[
 \sup_i\|Y\cdot\nabla \mathbf{F}_i\|_{L^\infty(\R^N)}<\infty.
\]
Indeed, $\mathbf{F}_i(y)=\mathbf{F}_*(y)$ for $|y|\geq2$, and
$|\nabla \mathbf{F}_*(y)|=O(|y|^{-m-1})$ and $Y(y)=O(|y|^2)$ as
$|y|\to\infty$, with $m\geq2$. Dividing \eqref{eq:KW} by $\varepsilon_i$, we
then obtain that
\begin{equation}\label{eq:normalized-kw}
 \int_{\R^N}(Y\cdot\nabla \mathbf{F}_i)u_i^{2^*}\,\mathrm{d}y=0,
\end{equation}
and it follows from \eqref{eq:tail-integral} that
\begin{equation}\label{eq:normalized-kw-tail}
 \left|\int_{\R^N\setminus \mathtt{V}}(Y\cdot\nabla \mathbf{F}_i)u_i^{2^*}\,\mathrm{d}y\right|
       \leq C_{\mathtt{V}} M_i^{-2^*}
\end{equation}
for each fixed neighborhood $\mathtt{V}$ of $\bar y$.

If $\nabla \mathbf{F}_*(\bar y)\ne0$, choose $Y(y)\equiv\nabla
\mathbf{F}_*(\bar y)$. On a fixed neighborhood $\mathtt{V}$ of $\bar y$, we then have
$Y\cdot\nabla \mathbf{F}_i\geq c>0$ for all large $i$. We deduce from
\eqref{eq:bubble-bounds} that $u_i\geq cM_i$ on $\mathbf
B_{M_i^{-2/(N-2)}}(y_i)\subset \mathtt{V}$. Hence
\[
 \int_{\mathtt{V}}(Y\cdot\nabla \mathbf{F}_i)u_i^{2^*}\,\mathrm{d}y
 \geq cM_i^{2^*}M_i^{-2N/(N-2)}=c>0,
\]
contradicting \eqref{eq:normalized-kw} and
\eqref{eq:normalized-kw-tail}. This proves that $\bar y=0$ or $|\bar y|=1$.

We now prove \eqref{eq:kw-local-upper}. Choose small $\rho>0$,
independent of $i$, such that
\[
 \overline{\mathtt{V}_\rho}\subset\mathcal V,\ \ 
 \mathtt{V}_\rho\subset\mathbf B_{1/2}(\bar y),
\]
and \eqref{eq:bubble-bounds} holds on $\mathtt{V}_\rho$ for all large $i$.
By Lemma~\ref{lem:supported-properties},
\[
 h_i=\tilde{h}_i=0,\ \ K_i=K_{\varepsilon_i}
       \,\ \ \text{on}\ \ \mathtt{V}_\rho.
\]

Set $r=|y|$. We distinguish two cases.

\noindent\textit{Case 1: $|\bar y|=1$.} For
$y\in\mathtt{V}_\rho\subset\mathbf B_{1/2}(\bar y)$, we have
\[
 \bar y\cdot y=1+\bar y\cdot(y-\bar y)
       \geq1-|y-\bar y|>\frac12.
\]
By \eqref{eq:equator-expansion} and \eqref{eq:equatorial-field},
\begin{equation}\label{eq:equatorial-sign}
 \begin{aligned}
 Y_{\bar y}\cdot\nabla K_i
 &=\frac{(\bar y\cdot y)(1-r^2)}{2r}K_{\varepsilon_i}'(r)\\
 &=\varepsilon_i m\left(\frac m2\right)^{m/2}(\bar y\cdot y)|r-1|^m
       +O(\varepsilon_i|r-1|^{m+1})\\
 &\geq c\varepsilon_i|r-1|^m,\ \ y\in\mathtt{V}_\rho.
 \end{aligned}
\end{equation}

\noindent\textit{Case 2: $\bar y=0$.} It follows from
\eqref{eq:polar-expansion} that
\begin{equation}\label{eq:polar-sign}
 y\cdot\nabla K_i=rK_{\varepsilon_i}'(r)
       \geq c\varepsilon_i|y|^m\geq c\varepsilon_i|y_N|^m,
       \ \ y\in\mathtt{V}_\rho.
\end{equation}

In both cases, it follows from \eqref{eq:equatorial-sign},
\eqref{eq:polar-sign} and the definition of $t_N$ that
\[
 Y\cdot\nabla K_i\geq c\varepsilon_i|t_N(y)|^m\geq0,
       \ \ y\in\mathtt{V}_\rho.
\]
Applying \eqref{eq:normalized-kw} and \eqref{eq:normalized-kw-tail} with
$\mathtt{V}=\mathtt{V}_\rho$, we obtain that
\[
 \begin{aligned}
 0\leq\int_{\mathtt{V}_\rho}(Y\cdot\nabla K_i)u_i^{2^*}\,\mathrm{d}y=-\varepsilon_i\int_{\R^N\setminus\mathtt{V}_\rho}
       (Y\cdot\nabla \mathbf{F}_i)u_i^{2^*}\,\mathrm{d}y\leq C\varepsilon_i M_i^{-2^*}.
 \end{aligned}
\]
This proves \eqref{eq:kw-local-upper}, which completes the proof.
\end{proof}

Combining Lemma~\ref{lem:kw-local} with the lower bound in
\eqref{eq:bubble-bounds}, we exclude the blow-up sequence in
Proposition~\ref{prop:blowup-sequence} and obtain the following upper bound.

\begin{lemma}\label{lem:uniform-upper}
Let $K$ satisfy \eqref{eq:admissible}.
There exists $C>0$, depending only on $N$ and $\Omega$, such that, for
$\varepsilon>0$ small enough, any solution $u$ of \eqref{eq:euclidean}
satisfies
\begin{equation}\label{eq:uniform-upper}
 u(y)\leq C(1+|y|^2)^{-(N-2)/2}\, \ \text{for any}\ \ y\in\R^N.
\end{equation}
\end{lemma}

\begin{proof}
We argue by contradiction. Suppose that there are
$0<\varepsilon_i\leq1/i$,
$h_i\in C_c^\infty(\Omega)$ and solutions $u_i$ of
\eqref{eq:euclidean} with $K=K_{\varepsilon_i}+h_i$ such that
$\|h_i\|_{C^m}\leq\varepsilon_i^2$ and
\[
 \sup_{y\in\R^N}(1+|y|^2)^{(N-2)/2}u_i(y)\geq i.
\]
By Proposition~\ref{prop:blowup-sequence}, up to a subsequence, $(u_i)$ has a
unique isolated simple blow-up point. If the blow-up point is at $\infty$,
we replace $(u_i,K_i)$ by $(\widehat u_i,\tilde{K}_i)$ and keep the notation
$(u_i,K_i)$. Let $y_i\to\bar y\in\R^N$ be the local maximum points in
Proposition~\ref{prop:blowup-sequence}, with $M_i=u_i(y_i)\to\infty$.
We take $Y$, $t$ and $\mathtt{V}_\rho$ as in Lemma~\ref{lem:kw-local}.
For $\rho>0$ small, the inverse $t^{-1}$ satisfies
\[
 |t^{-1}(z)-t^{-1}(w)|\leq C|z-w|,\ \ 
 |\det D(t^{-1})(z)|\geq c,\ \ z,w\in\mathbf B_\rho,
\]
where $c>0$ and $C\geq1$ are independent of $i$.
Since $t(y_i)\to0$ and $M_i\to\infty$, for all large $i$ we have
\[
 \mathbf B_{\frac1C M_i^{-2/(N-2)}}(t(y_i))
 \subset t(\mathbf B_{M_i^{-2/(N-2)}}(y_i))
 \subset\mathbf B_\rho.
\]
Hence $t(\mathbf B_{M_i^{-2/(N-2)}}(y_i))$ contains a rectangular box
$Q_i=\mathcal I_{i,1}\times\cdots\times\mathcal I_{i,N}$ with
\[
 cM_i^{-2/(N-2)}\leq|\mathcal I_{i,\ell}|\leq CM_i^{-2/(N-2)},
 \ \ 1\leq\ell\leq N.
\]
We have
\[
 \begin{aligned}
 \int_{\mathcal I_{i,N}}|s|^m\,\mathrm{d}s
 \geq\int_{-\frac{|\mathcal I_{i,N}|}{2}}^{\frac{|\mathcal I_{i,N}|}{2}}
       |s|^m\,\mathrm{d}s=\frac{|\mathcal I_{i,N}|^{m+1}}{2^m(m+1)}
 \geq cM_i^{-2(m+1)/(N-2)}.
 \end{aligned}
\]
Making the change of variables $z=t(y)$, we then obtain that
\[
 \begin{aligned}
 \int_{t^{-1}(Q_i)}|t_N(y)|^m\,\mathrm{d}y
 &=\int_{Q_i}|z_N|^m|\det D(t^{-1})(z)|\,\mathrm{d}z\\
 &\geq c\left(\prod_{\ell=1}^{N-1}|\mathcal I_{i,\ell}|\right)
             \int_{\mathcal I_{i,N}}|s|^m\,\mathrm{d}s\\
 &\geq cM_i^{-2(N+m)/(N-2)}.
 \end{aligned}
\]
Since $t^{-1}(Q_i)\subset\mathbf B_{M_i^{-2/(N-2)}}(y_i)$ and
$u_i\geq cM_i$ on $\mathbf B_{M_i^{-2/(N-2)}}(y_i)$, we deduce from
Lemma~\ref{lem:kw-local} that
\[
 \int_{\mathtt{V}_\rho}(Y\cdot\nabla K_i)u_i^{2^*}\,\mathrm{d}y
 \geq c\varepsilon_i M_i^{2^*-2(N+m)/(N-2)}
 =c\varepsilon_i M_i^{-2m/(N-2)}.
\]
By \eqref{eq:kw-local-upper}, we derive that
\[
 M_i^{2(N-m)/(N-2)}\leq C.
\]
Since $m=N-2<N$ and $M_i\to\infty$ as $i\to\infty$,
we obtain a contradiction.
This completes the proof.
\end{proof}

We now use Lemma~\ref{lem:uniform-upper} to complete the a-priori estimates
for $u$ and $\widehat u$. The following proposition also gives compactness
for solutions with $h$ in a fixed finite-dimensional subspace of
$C_c^\infty(\Omega)$.

\begin{proposition}\label{prop:compactness}
Fix $0<\alpha<1$. Suppose that $\varepsilon>0$ is small enough
and $K$ satisfies \eqref{eq:admissible}.
Then there exist constants $c,C>0$, depending only on $N$, $\Omega$
and $\alpha$, such that any solution $u$ of \eqref{eq:euclidean} satisfies
\begin{equation}\label{eq:compact-bounds}
 c(1+|y|^2)^{-(N-2)/2}\leq u(y)
 \leq C(1+|y|^2)^{-(N-2)/2},\ \ y\in\R^N,
\end{equation}
and
\[
 \|u\|_{C^{3,\alpha}(\overline{\mathbf B}_2)}
 +\|\widehat u\|_{C^{3,\alpha}(\overline{\mathbf B}_2)}\leq C.
\]

Moreover, suppose $h_i\to h$ as
$i\to\infty$ in a fixed finite-dimensional subspace of $C_c^\infty(\Omega)$
and that $K_{\varepsilon}+h_i$ satisfies \eqref{eq:admissible} for all $i$.
If $u_i$ solves \eqref{eq:euclidean} with coefficient $K_{\varepsilon}+h_i$
for each $i$, then, up to a subsequence,
\[
 u_i\to u,\ \ \widehat u_i\to\widehat u
 \,\ \ \text{in}\ \ C^{2,\alpha}(\overline{\mathbf B}_2)
 \ \ \text{as}\ \ i\to\infty,
\]
where $u$ solves \eqref{eq:euclidean} with coefficient $K_{\varepsilon}+h$.
\end{proposition}

\begin{proof}
The upper bound in \eqref{eq:compact-bounds} follows from
Lemma~\ref{lem:uniform-upper}. Multiplying \eqref{eq:euclidean} by $u$ and
integrating over $\R^N$, we obtain that
\begin{equation}\label{eq:energy-lower-bound}
 \int_{\R^N}|\nabla u|^2\,\mathrm{d}y
 =\int_{\R^N}Ku^{2^*}\,\mathrm{d}y
 \leq C\|u\|_{L^{2^*}(\R^N)}^{2^*},\ \ 
 \|u\|_{L^{2^*}(\R^N)}\geq c>0.
\end{equation}
By Lemma~\ref{lem:supported-properties}, \eqref{eq:kelvin-solution} and
\eqref{eq:uniform-upper}, $Ku^{2^*-2}$ and $\tilde{K}\widehat u^{2^*-2}$ are
uniformly bounded on $\overline{\mathbf B}_3$. Applying Harnack's inequality
to $u$ and $\widehat u$ in $\mathbf B_2$, we obtain that
\begin{equation}\label{eq:weighted-harnack}
 \sup_{y\in\R^N}(1+|y|^2)^{(N-2)/2}u(y)
 \leq C\inf_{y\in\R^N}(1+|y|^2)^{(N-2)/2}u(y).
\end{equation}
It follows from \eqref{eq:energy-lower-bound} and \eqref{eq:weighted-harnack} that the lower bound in
\eqref{eq:compact-bounds} holds.

Applying elliptic estimates to \eqref{eq:euclidean} and
$-\Delta\widehat u=\tilde{K}\widehat u^{2^*-1}$, we then deduce from
Lemma~\ref{lem:supported-properties} and \eqref{eq:compact-bounds} that
\[
 \|u\|_{C^{3,\alpha}(\overline{\mathbf B}_2)}
 +\|\widehat u\|_{C^{3,\alpha}(\overline{\mathbf B}_2)}\leq C.
\]
Thus, when $h_i\to h$ in the fixed finite-dimensional subspace of
$C_c^\infty(\Omega)$, up to a subsequence, we have
\[
 \|u_i-u\|_{C^{2,\alpha}(\overline{\mathbf B}_2)}
 +\|\widehat u_i-\widehat u\|_{C^{2,\alpha}(\overline{\mathbf B}_2)}\to0
 \ \ \text{as}\ \ i\to\infty.
\]
By \eqref{eq:kelvin-solution}, we also have $u_i\to u$ in
$C^2_{\rm loc}(\R^N)$. Since $h_i\to h$ in $C^m(\R^N)$, we pass to the limit in the equations
$-\Delta u_i=(K_{\varepsilon}+h_i)u_i^{2^*-1}$ to obtain that
\[
 -\Delta u=(K_{\varepsilon}+h)u^{2^*-1}\,\ \ \text{in}\ \ \R^N.
\]
Moreover, it follows from \eqref{eq:compact-bounds} that
$u(y)\geq c(1+|y|^2)^{-(N-2)/2}>0$ for $y\in\R^N$.
This completes the proof.
\end{proof}

\section{Proof of Theorem~\ref{thm:construction}}

The purpose of this section is to complete the proof of
Theorem~\ref{thm:construction}.
Throughout this section, we fix $\varepsilon>0$ small enough and consider $h\in C_c^\infty(\Omega)$ with
$\|h\|_{C^m}\leq\varepsilon^2$, as in
Lemma~\ref{lem:supported-properties} and
Proposition~\ref{prop:compactness}.
We first prove that the linearized operator of
\eqref{eq:euclidean} around a solution $u$ is Fredholm of index zero
and characterize its kernel and range.
We then choose a finite-dimensional subspace of $C_c^\infty(\Omega)$
and apply Sard's theorem to the projection $(h,u)\mapsto h$
on the corresponding solution manifold to obtain
an $h$ for which all solutions of \eqref{eq:euclidean} with
$K=K_{\varepsilon}+h$ are non-degenerate in $D^{1,2}(\R^N)$.
Using the implicit function theorem, we thus deduce from
Proposition~\ref{prop:compactness} that the solution set of
\eqref{eq:euclidean} is finite.

Fix $0<\alpha<1$ and set $p=2^*-1$. Define for any $y\ne0$,
\begin{equation}\label{eq:space-transforms}
 \widehat\phi(y):=|y|^{2-N}\phi\left(\frac{y}{|y|^2}\right),
 \ \ (\mathcal T f)(y):=|y|^{-N-2}f\left(\frac{y}{|y|^2}\right).
\end{equation}
We also define
\begin{equation}\label{eq:function-spaces}
 \begin{aligned}
 \mathcal X
 &:=\left\{\phi\in C^{2,\alpha}_{\rm loc}(\R^N):\,
       \widehat\phi\ \ \text{has a}\ \ C^{2,\alpha}
       \ \ \text{extension across}\ \ 0\right\},\\
 \mathcal Y
 &:=\left\{f\in C^{0,\alpha}_{\rm loc}(\R^N):\,
       \mathcal T f\ \ \text{has a}\ \ C^{0,\alpha}
       \ \ \text{extension across}\ \ 0\right\},
 \end{aligned}
\end{equation}
with the norms
\begin{align}
 \|\phi\|_{\mathcal X}
 &:=\|\phi\|_{C^{2,\alpha}(\overline{\mathbf{B}}_2)}
   +\|\widehat\phi\|_{C^{2,\alpha}(\overline{\mathbf{B}}_2)},
       \label{eq:chart-norm}\\
 \|f\|_{\mathcal Y}
 &:=\|f\|_{C^{0,\alpha}(\overline{\mathbf{B}}_2)}
   +\|\mathcal T f\|_{C^{0,\alpha}(\overline{\mathbf{B}}_2)},
       \label{eq:forcing-space}
\end{align}
respectively.
The pairs $(\phi,\widehat\phi)$ and $(f,\mathcal T f)$ form
closed subspaces of
$C^{2,\alpha}(\overline{\mathbf{B}}_2)^2$ and
$C^{0,\alpha}(\overline{\mathbf{B}}_2)^2$, respectively,
since the limit of any sequence of such pairs that converges in these spaces
also satisfies the identities \eqref{eq:space-transforms}
on $1/2\leq|y|\leq2$.
Thus $\mathcal X$ and $\mathcal Y$ are Banach spaces.
Define
\[
 \mathcal X_+
 :=\left\{u\in\mathcal X:\,
       u>0,\ \ \widehat u>0
       \,\ \ \text{on}\ \ \overline{\mathbf{B}}_2\right\}.
\]
For any $u\in\mathcal X_+$, both $u$ and $\widehat u$
are bounded below by a positive constant on
$\overline{\mathbf{B}}_2$.
It then follows from \eqref{eq:chart-norm} that
$\mathcal X_+$ is open in $\mathcal X$.
We also deduce from \eqref{eq:solution-regularity} that
all solutions of \eqref{eq:euclidean} belong to $\mathcal X_+$.

\begin{lemma}\label{lem:function-spaces}
Suppose that $\mathcal X$ and $\mathcal Y$ are as in
\eqref{eq:function-spaces}, with the norms \eqref{eq:chart-norm}
and \eqref{eq:forcing-space}, respectively.
Then the embeddings
$\mathcal X\hookrightarrow D^{1,2}(\R^N)$ and
$\mathcal Y\hookrightarrow L^1(\R^N)\cap L^{2N/(N+2)}(\R^N)$
are continuous.
Moreover, the operator $-\Delta:\mathcal X\to\mathcal Y$
is an isomorphism, and its inverse
$G=(-\Delta)^{-1}:\mathcal Y\to\mathcal X$
is given by
\begin{equation}\label{eq:newton-potential}
 Gf(y)=\frac1{N(N-2)|\mathbf{B}_1|}
       \int_{\R^N}|y-z|^{2-N}f(z)\,\mathrm{d}z,
       \ \ f\in\mathcal Y,
\end{equation}
where $|\mathbf{B}_1|$ denotes the volume of the unit ball.
\end{lemma}

\begin{proof}
It follows from \eqref{eq:space-transforms} and \eqref{eq:chart-norm} that
for $\phi\in\mathcal X$,
\[
 |\phi(y)|\leq C\|\phi\|_{\mathcal X}(1+|y|)^{2-N},
 \ \ |\nabla\phi(y)|\leq C\|\phi\|_{\mathcal X}(1+|y|)^{1-N}.
\]
Multiplying $\phi$ by cut-off functions supported in large balls
and taking the convolution of the resulting functions with a smooth mollifier,
we then obtain that
$\mathcal X\hookrightarrow D^{1,2}(\R^N)$
with $\|\nabla\phi\|_{L^2}\leq C\|\phi\|_{\mathcal X}$.
We also deduce from \eqref{eq:space-transforms} and
\eqref{eq:forcing-space} that, for $f\in\mathcal Y$,
\[
 |f(y)|\leq C\|f\|_{\mathcal Y}(1+|y|)^{-N-2},
\]
and hence
\[
 \|f\|_{L^1(\R^N)}+\|f\|_{L^{2N/(N+2)}(\R^N)}
 \leq C\|f\|_{\mathcal Y}.
\]
This proves the continuity of both embeddings in
Lemma~\ref{lem:function-spaces}.

It follows from \eqref{eq:space-transforms} that
$\mathcal T(-\Delta\phi)=-\Delta\widehat\phi$ for any $\phi\in\mathcal X$.
Combining this identity with \eqref{eq:chart-norm} and
\eqref{eq:forcing-space}, we then obtain that
$-\Delta:\mathcal X\to\mathcal Y$ is bounded.
Note from \eqref{eq:space-transforms} and \eqref{eq:forcing-space} that
$\mathcal T^2f=f$, $\mathcal T f\in\mathcal Y$ and
$\|\mathcal T f\|_{\mathcal Y}=\|f\|_{\mathcal Y}$.
We next prove that $Gf\in\mathcal X$ and
\begin{equation}\label{eq:newton-bounds}
 \begin{gathered}
 |Gf(y)|+(1+|y|)|\nabla Gf(y)|
 \leq C\|f\|_{\mathcal Y}(1+|y|)^{2-N},\\
 \widehat{Gf}=G(\mathcal T f),
 \ \ \|Gf\|_{\mathcal X}\leq C\|f\|_{\mathcal Y}.
 \end{gathered}
\end{equation}
Set $r=|y|$. If $r\geq2$, split the integral in
\eqref{eq:newton-potential} into $\mathbf B_{r/2}(y)$ and
$\R^N\setminus\mathbf B_{r/2}(y)$. On $\mathbf B_{r/2}(y)$,
we have $|z|\geq r/2$ and hence
$|f(z)|\leq C\|f\|_{\mathcal Y}r^{-N-2}$.
On $\R^N\setminus\mathbf B_{r/2}(y)$, $|y-z|\geq r/2$. Thus, for $j=0,1$,
\[
 \begin{aligned}
 \int_{\R^N}|y-z|^{2-N-j}|f(z)|\,\mathrm{d}z
 &\leq C\|f\|_{\mathcal Y}r^{-N-2}
       \int_0^{r/2}\rho^{1-j}\,\mathrm{d}\rho
       +Cr^{2-N-j}\|f\|_{L^1(\R^N)}\\
 &\leq C\|f\|_{\mathcal Y}r^{2-N-j}.
 \end{aligned}
\]
If $r\leq2$, we split the integral into $\mathbf B_1(y)$ and
$\R^N\setminus\mathbf B_1(y)$. For $j=0,1$, we obtain that
\[
 \int_{\R^N}|y-z|^{2-N-j}|f(z)|\,\mathrm{d}z
 \leq C\|f\|_{L^\infty(\R^N)}+\|f\|_{L^1(\R^N)}
 \leq C\|f\|_{\mathcal Y}.
\]
Since $|\nabla_y|y-z|^{2-N}|=(N-2)|y-z|^{1-N}$,
we deduce from \eqref{eq:newton-potential} that
$|Gf(y)|\leq C\|f\|_{\mathcal Y}(1+|y|)^{2-N}$ and
$|\nabla Gf(y)|\leq C\|f\|_{\mathcal Y}(1+|y|)^{1-N}$.
This proves the first inequality in \eqref{eq:newton-bounds}.
Substituting $z=t/|t|^2$ in \eqref{eq:newton-potential},
we obtain that $\widehat{Gf}=G(\mathcal T f)$ for $y\ne0$.
If $1\leq|y|\leq3$, then
$y/|y|^2\in\overline{\mathbf B}_1$.
All derivatives of $y\mapsto y/|y|^2$ and $y\mapsto|y|^{-N-2}$
are bounded on $\{y:\,1\leq|y|\leq3\}$.
We therefore deduce from \eqref{eq:space-transforms}
and \eqref{eq:forcing-space} that
\[
 \|f\|_{C^{0,\alpha}(\overline{\mathbf{B}}_3)}
 +\|\mathcal T f\|_{C^{0,\alpha}(\overline{\mathbf{B}}_3)}
 \leq C\|f\|_{\mathcal Y}.
\]
The Newtonian potential in \eqref{eq:newton-potential} satisfies $-\Delta Gf=f$
in the sense of distributions in $\R^N$.
Applying the Schauder estimates on $\mathbf{B}_3$ to $Gf$ and $G(\mathcal T f)$,
we then obtain bounds for their $C^{2,\alpha}(\overline{\mathbf{B}}_2)$ norms.
Since $-\Delta G(\mathcal T f)=\mathcal T f$ and $\mathcal T f\in C^{0,\alpha}$
near $0$, we use $\widehat{Gf}=G(\mathcal T f)$ to extend $\widehat{Gf}$ across $0$.
Thus $Gf\in\mathcal X$ and $\|Gf\|_{\mathcal X}\leq C\|f\|_{\mathcal Y}$,
which proves \eqref{eq:newton-bounds}.
If $\phi\in\mathcal X$ satisfies $-\Delta\phi=0$ in
$\R^N$, then $\phi(y)\to0$ as $|y|\to\infty$.
Liouville's theorem therefore implies that $\phi=0$.
Hence $-\Delta:\mathcal X\to\mathcal Y$ is an isomorphism with inverse $G$,
which completes the proof of Lemma~\ref{lem:function-spaces}.

\end{proof}

We next apply Lemma~\ref{lem:function-spaces} to prove that the linearized
operator of \eqref{eq:euclidean} is Fredholm of index zero and to characterize
its kernel and range.

\begin{lemma}\label{lem:linearization}
Suppose $K\in C^{0,\alpha}_{\rm loc}(\R^N)$ is positive
and $\tilde{K}$ has a $C^{0,\alpha}$ extension across $0$
with $\tilde{K}(0)>0$.
Let $u$ be a solution of \eqref{eq:euclidean}.
Then the operator
\[
 A_u:=-\Delta-pKu^{p-1}:\mathcal X\to\mathcal Y
\]
is Fredholm of index zero.
Moreover, $\ker A_u$ consists exactly of the weak solutions
$\phi\in D^{1,2}(\R^N)$ of
$-\Delta\phi=pKu^{p-1}\phi$ in $\R^N$, and
\begin{equation}\label{eq:range-criterion}
 \operatorname{ran}A_u
 =\left\{f\in\mathcal Y:\,
       \int_{\R^N}f\phi\,\mathrm{d}y=0
       \, \ \text{for any}\ \ \phi\in\ker A_u\right\}.
\end{equation}
\end{lemma}

\begin{proof}
Set
\[
 V=pKu^{p-1},
 \ \ \tilde{V}(y)=|y|^{-4}V\left(\frac{y}{|y|^2}\right)
       =p\tilde{K}(y)\widehat u(y)^{p-1}.
\]
Applying the regularity argument in Lemma~\ref{lem:regularity} under the
present assumptions on $K$ and $\tilde{K}$, we then obtain that
$u,\widehat u\in C^{2,\alpha}_{\rm loc}(\R^N)$ and
$u,\widehat u>0$ in $\R^N$.
Thus $V,\tilde{V}\in C^{0,\alpha}_{\rm loc}(\R^N)$, $V(y)=O((1+|y|)^{-4})$ as
$|y|\to\infty$ and $V\in L^{N/2}(\R^N)$.
Define $M_V\phi:=V\phi$ for $\phi\in\mathcal X$.
We deduce from \eqref{eq:space-transforms} that
$\mathcal T(M_V\phi)=\tilde{V}\widehat\phi$, and hence
$M_V\phi\in\mathcal Y$.
Applying the compact embedding
$C^{2,\alpha}(\overline{\mathbf{B}}_2)\hookrightarrow
C^{0,\alpha}(\overline{\mathbf{B}}_2)$ to $\phi$ and $\widehat\phi$,
we then obtain that $M_V:\mathcal X\to\mathcal Y$ is compact.
It follows from Lemma~\ref{lem:function-spaces} that the operator
\begin{equation*}
 A_u=-\Delta-V=(-\Delta)(I-GM_V):\mathcal X\to\mathcal Y
\end{equation*}
is Fredholm of index zero.

Suppose $\phi\in D^{1,2}(\R^N)$ satisfies $-\Delta\phi=V\phi$ weakly in $\R^N$.
Then its Kelvin transform $\widehat\phi$ belongs to $D^{1,2}(\R^N)$
and satisfies
\begin{equation}\label{eq:kelvin-linearized}
 -\Delta\widehat\phi=\tilde{V}\widehat\phi
 \,\ \ \text{in}\ \ \R^N\setminus\{0\}.
\end{equation}
For any $\zeta\in C_c^\infty(\R^N)$, we test
\eqref{eq:kelvin-linearized} with $\chi_\sigma\zeta$, where
$\chi_\sigma$ is a smooth cut-off function with $0\leq\chi_\sigma\leq1$,
vanishing on $\mathbf B_\sigma$, equal to $1$ outside $\mathbf B_{2\sigma}$,
and satisfying $\|\nabla\chi_\sigma\|_{L^2}\leq C\sigma^{(N-2)/2}$.
Letting $\sigma\to0^{+}$, we obtain that
\eqref{eq:kelvin-linearized} holds weakly in $\R^N$.
Since $V,\tilde{V}\in C^{0,\alpha}_{\rm loc}(\R^N)$,
we apply a bootstrap argument and Schauder estimates to
$-\Delta\phi=V\phi$ and \eqref{eq:kelvin-linearized}
to obtain that
$\phi,\widehat\phi\in C^{2,\alpha}_{\rm loc}(\R^N)$
and hence $\phi\in\mathcal X$.
Conversely, it follows from Lemma~\ref{lem:function-spaces} that
$\mathcal X\subset D^{1,2}(\R^N)$. Thus $\phi\in\ker A_u$ if and only if
$\phi\in D^{1,2}(\R^N)$ and $-\Delta\phi=V\phi$ holds weakly in $\R^N$.
Define
\[
 Z:=\left\{f\in\mathcal Y:\,
       \int_{\R^N}f\phi\,\mathrm{d}y=0
       \, \ \text{for any}\ \ \phi\in\ker A_u\right\}.
\]
For $\zeta\in\mathcal X$ and $\phi\in\ker A_u$,
we test $-\Delta\phi=V\phi$ with $\zeta$ to obtain that
\[
 \int_{\R^N}(A_u\zeta)\phi\,\mathrm{d}y
 =\int_{\R^N}\bigl(\nabla\zeta\cdot\nabla\phi-V\zeta\phi\bigr)\,\mathrm{d}y=0.
\]
Since $\zeta,\phi\in D^{1,2}(\R^N)$,
$A_u\zeta\in\mathcal Y\subset L^{2N/(N+2)}(\R^N)$ and
$V\in L^{N/2}(\R^N)$, we apply H\"older's inequality to obtain that
$(A_u\zeta)\phi,\ \ \nabla\zeta\cdot\nabla\phi,\ \ V\zeta\phi
\in L^1(\R^N)$. Thus $\operatorname{ran}A_u\subset Z$.
Set $d:=\dim\ker A_u$ and choose a basis
$\phi_1,\ldots,\phi_d$ of $\ker A_u$.
Since $C_c^\infty(\R^N)\subset\mathcal Y$, the functionals
$f\mapsto\int_{\R^N}f\phi_j\,\mathrm{d}y$, $j=1,\ldots,d$, are linearly
independent in $\mathcal Y^*$, so that $Z$ has codimension $d$.
The operator $A_u$ is Fredholm of index zero, and therefore
$\operatorname{ran}A_u$ also has codimension $d$.
Using $\operatorname{ran}A_u\subset Z$, we then obtain that
$\operatorname{ran}A_u=Z$, which proves \eqref{eq:range-criterion}.
This completes the proof of Lemma~\ref{lem:linearization}.

\end{proof}

We remark that a solution $u$ of \eqref{eq:euclidean} is non-degenerate in
$D^{1,2}(\R^N)$ if $\ker A_u=\{0\}$. We now use
Lemma~\ref{lem:linearization} and Sard's theorem to prove
Theorem~\ref{thm:construction}.

\begin{proof}[Proof of Theorem~\ref{thm:construction}]
Let $\mathcal S_0$ be the set of all solutions of \eqref{eq:euclidean} with
$K=K_{\varepsilon}$. It follows from Proposition~\ref{prop:compactness} that
$\mathcal S_0$ is compact in $\mathcal X$.
If $\mathcal S_0=\varnothing$, the conclusion of
Theorem~\ref{thm:construction} holds with $h=0$.
We therefore assume that $\mathcal S_0\ne\varnothing$.
For $h\in C_c^\infty(\Omega)$ and $u\in\mathcal X_+$, define
\begin{equation*}
 F(h,u):=-\Delta u-(K_{\varepsilon}+h)u^p\in\mathcal Y.
\end{equation*}
If $K=K_{\varepsilon}+h$ satisfies \eqref{eq:admissible}, it follows from
Lemmas~\ref{lem:regularity} and~\ref{lem:function-spaces} that the zero points of
$F(h,\cdot)$ in $\mathcal X_+$ are exactly the solutions of
\eqref{eq:euclidean}.
Let $E'$ be any finite-dimensional subspace of $C_c^\infty(\Omega)$.
Using $\mathcal T(Ku^p)=\tilde{K}\widehat u^p$ for $K=K_{\varepsilon}+h$ with
$h\in E'$ and the inequalities $\min_{\overline{\mathbf{B}}_2}u>0$ and
$\min_{\overline{\mathbf{B}}_2}\widehat u>0$, we then obtain that
$F:E'\times\mathcal X_+\to\mathcal Y$ is smooth.
For $(k,\phi)\in E'\times\mathcal X$, we calculate that
\[
 DF(h,u)[k,\phi]
 =-\Delta\phi-p(K_{\varepsilon}+h)u^{p-1}\phi-ku^p.
\]

For any $u\in\mathcal S_0$, we use the operator $A_u$ defined in
Lemma~\ref{lem:linearization} with $K=K_{\varepsilon}$.
Suppose that $\phi\in\ker A_u$ satisfies
\begin{equation}\label{eq:cokernel}
 \int_{\R^N}h u^p\phi\,\mathrm{d}y=0
       \, \ \text{for any}\ \ h\in C_c^\infty(\Omega).
\end{equation}
Since $u>0$, we deduce from \eqref{eq:cokernel} that $\phi=0$ in $\Omega$.
Applying unique continuation \cite{Aronszajn}, we then obtain that
$\phi=0$ in $\R^N$.
It follows from \eqref{eq:range-criterion} that
$\mathcal Y/\operatorname{ran}A_u$ is spanned by the classes of $hu^p$
with $h\in C_c^\infty(\Omega)$;
otherwise, there exists a nonzero $\phi\in\ker A_u$ satisfying
\eqref{eq:cokernel}.
Hence the map
\[
 C_c^\infty(\Omega)\times\mathcal X\longrightarrow\mathcal Y,
 \ \ (h,\phi)\longmapsto A_u\phi-hu^p
\]
is onto.
Since $\operatorname{ran}A_u$ has finite codimension in $\mathcal Y$,
we can choose a finite-dimensional subspace $E_u$ of $C_c^\infty(\Omega)$
such that $DF(0,u):E_u\times\mathcal X\to\mathcal Y$ is onto.

For any $u_0\in\mathcal S_0$, the operator
$DF(0,u_0):E_{u_0}\times\mathcal X\to\mathcal Y$ is onto and has a
finite-dimensional kernel, so that it has a bounded linear right inverse.
Using the continuity of $DF$, we then obtain
a neighborhood $\mathcal U_{u_0}\subset\mathcal X_+$ of $u_0$ such that
$DF(0,u):E_{u_0}\times\mathcal X\to\mathcal Y$ remains onto
whenever $u\in\mathcal U_{u_0}$.
Since $\mathcal S_0$ is compact in $\mathcal X$, we can choose
$u_1,\ldots,u_L\in\mathcal S_0$ such that
\[
 \mathcal S_0\subset\bigcup_{j=1}^L\mathcal U_{u_j}.
\]
Define
\[
 E:=E_{u_1}+\cdots+E_{u_L}.
\]
If $E=\{0\}$, replace $E$ by $\operatorname{span}\{k_0\}$
for a nonzero $k_0\in C_c^\infty(\Omega)$.
For any $u\in\mathcal S_0$, choose $j$ such that
$u\in\mathcal U_{u_j}$.
Since $E_{u_j}\subset E$, we obtain that
$DF(0,u):E\times\mathcal X\to\mathcal Y$ is onto.
This operator also has a finite-dimensional kernel and hence a bounded
linear right inverse $\mathcal R:\mathcal Y\to E\times\mathcal X$.
Using the continuity of $DF$, we then obtain that
\[
 \|(DF(h,v)-DF(0,u))\mathcal R\|_{\mathcal L(\mathcal Y,\mathcal Y)}<1
\]
for $(h,v)$ in a neighborhood of $(0,u)$ in $E\times\mathcal X_+$.
The Neumann series then gives the inverse of
\[
 DF(h,v)\mathcal R
 =I_{\mathcal Y}+(DF(h,v)-DF(0,u))\mathcal R,
\]
so that $DF(h,v):E\times\mathcal X\to\mathcal Y$ is onto.

Restrict $F$ to $E\times\mathcal X_+$ and choose $0<\tau\leq\varepsilon^2$
such that $K_{\varepsilon}+h$ satisfies \eqref{eq:admissible} whenever
$h\in E$ and $\|h\|_{C^m}<\tau$.
We next prove that after decreasing $\tau$ if necessary,
$DF(h,u):E\times\mathcal X\to\mathcal Y$ is onto whenever $F(h,u)=0$ and
$\|h\|_{C^m}<\tau$.
On the contrary, suppose there exist zero points $(h_i,u_i)$ of $F$ with $h_i\in E$,
$\|h_i\|_{C^m}\to0$ as $i\to\infty$ and $DF(h_i,u_i)$ not onto.
It follows from Proposition~\ref{prop:compactness} that up to a subsequence,
$(h_i,u_i)\to(0,u)$ in $E\times\mathcal X$ as $i\to\infty$ for some
$u\in\mathcal S_0$. This contradicts the surjectivity of $DF$ near $(0,u)$
and proves the asserted choice of $\tau$.
Since $E$ is finite-dimensional and $D_uF(h,u)$ has Fredholm index zero
at each zero point $(h,u)$ of $F$, we obtain that $DF(h,u)$ has
Fredholm index $\dim E$.
Applying the implicit function theorem, we then conclude that the set
\[
 \mathcal M=\{(h,u)\in E\times\mathcal X_+:\,\|h\|_{C^m}<\tau,\ \ F(h,u)=0\}
\]
is a smooth manifold of dimension $\dim E$.

Consider the map $\pi:\mathcal M\to\{h\in E:\,\|h\|_{C^m}<\tau\}$ given by
$\pi(h,u)=h$.
Since $E$ is finite-dimensional, we deduce from Proposition~\ref{prop:compactness}
that $\{(h,u)\in\mathcal M:\,\|h\|_{C^m}\leq r\}$ is compact in
$E\times\mathcal X$ for any $r\in(0,\tau)$.
Choose $r_\ell\in(0,\tau)$ with $r_\ell\to\tau^{-}$ as $\ell\to\infty$
such that the compact sets
$\{(h,u)\in\mathcal M:\,\|h\|_{C^m}\leq r_\ell\}$
cover $\mathcal M$.
Since each of these compact sets has a finite cover by coordinate
neighborhoods, we obtain a countable cover
$\mathcal M=\bigcup_{j=1}^{\infty}U_j$.
It follows from Sard's theorem \cite{Sard} that the critical values of $\pi$
have measure zero in $E$. Thus there exist $h\in E$ arbitrarily close to zero
such that
\[
 D\pi(h,u)\bigl(T_{(h,u)}\mathcal M\bigr)=E
 \, \ \text{for any}\ \ u\ \ \text{with}\ \ (h,u)\in\mathcal M.
\]
For any such $h$ and any $(h,u)\in\mathcal M$, we apply
Lemma~\ref{lem:linearization} with $K=K_{\varepsilon}+h$, so that
$D_uF(h,u)=A_u$. Since
\[
 T_{(h,u)}\mathcal M=\{(k,\phi)\in E\times\mathcal X:\,A_u\phi-ku^p=0\},
\]
we use the surjectivity of $D\pi(h,u)$ to obtain that
$ku^p\in\ran A_u$ for any $k\in E$.
Moreover, $DF(h,u)[k,\phi]=A_u\phi-ku^p$ maps
$E\times\mathcal X$ onto $\mathcal Y$, so that $\ran A_u=\mathcal Y$.
Since $A_u$ has Fredholm index zero, we have $\ker A_u=\{0\}$.
We therefore conclude from Lemma~\ref{lem:linearization} that if
$h\in E$ satisfies $\|h\|_{C^m}<\tau$ and is outside the critical values of $\pi$,
then all solutions of \eqref{eq:euclidean} with $K=K_{\varepsilon}+h$ are
non-degenerate in $D^{1,2}(\R^N)$.

Since all norms on $E$ are equivalent, we can choose $h\in E$ outside the
critical values of $\pi$ such that $\|h\|_{C^m}<\tau$ and
$\|h\|_{C^s}<\eta$.
For $K=K_{\varepsilon}+h$, it follows from
Lemma~\ref{lem:supported-properties} that $K\in C^2(\R^N)$,
$K>0$ in $\R^N$ and \eqref{eq:maxset} holds. Since $h=0$ near $\{|y|=1\}$, the expansion in
\eqref{eq:equator-expansion} yields that \eqref{eq:euclidean-expansion} holds with
$c_0=\varepsilon(m/2)^{m/2}$.
For $K=K_{\varepsilon}+h$, we deduce from Proposition~\ref{prop:compactness} that the
solution set of \eqref{eq:euclidean} is compact in $\mathcal X$.
Since $D_uF(h,u):\mathcal X\to\mathcal Y$ is an isomorphism at each solution
$u$, the implicit function theorem implies that $u$ is the only solution of
\eqref{eq:euclidean} in some neighborhood of $u$ in $\mathcal X$.
Suppose that \eqref{eq:euclidean} admits infinitely many distinct
solutions $u_j$. It follows from Proposition~\ref{prop:compactness}
that up to a subsequence, $u_j\to u_*$ in $\mathcal X$ as
$j\to\infty$ for some solution $u_*$.
The local uniqueness at $u_*$ then implies that $u_j=u_*$
for all large $j$, contradicting $u_j\ne u_k$ for $j\ne k$.
Hence the solution set of \eqref{eq:euclidean} is finite, which
completes the proof of Theorem~\ref{thm:construction}.

\end{proof}

\section*{Acknowledgements}

The second author is grateful to Liping Wang and Chunyi Zhao for drawing his
attention to Wei--Yan's conjecture during discussions at their seminar.

\section*{Declaration of AI-assisted tools}

This paper was written by the authors and was not generated by artificial
intelligence. During its preparation, the authors used AI tools as auxiliary
aids for exploratory discussions, preliminary proof drafting, and language
editing. The authors take full responsibility for the mathematical validation,
the final proof choices, and the final text.

\end{document}